\documentclass[11pt,english,a4paper]{amsart}

\usepackage{hyperref}
\usepackage{babel}
\usepackage{fancyhdr,amsfonts,amssymb,latexsym,xspace,epsfig,graphicx,color}
\usepackage{amsmath,enumerate}
\usepackage{numprint}
\usepackage{asymptote}

\usepackage[square,numbers]{natbib}

\numberwithin{equation}{section}

\usepackage[latin1]{inputenc}
\usepackage{ifpdf}
\ifpdf 
\DeclareGraphicsExtensions{.pdf,.pdftex_t,.pstex_t,.pstex,.png,.jpg}
\else 
\DeclareGraphicsExtensions{.eps,.ps,.pst}
\fi
\definecolor{orange}{rgb}{1,0.6,0.05}
\definecolor{purple}{rgb}{0.6,0,1}

\newtheorem{theorem}{Theorem}[section]

\newtheorem{prop}[theorem]{Proposition}

\newtheorem{lemma}[theorem]{Lemma}

\newtheorem{remark}[theorem]{Remark}

\newtheorem{defi}[theorem]{Definition}

\newtheorem{ex}[theorem]{Example}

\newcommand{\R}{\ensuremath{\mathbb{R}}}

\renewcommand{\P}{\ensuremath{\mathbb{P}}}

\newcommand{\E}{\ensuremath{\mathbb{E}}}

\newcommand{\cov}{\mbox{\rm Cov}}

\newcommand{\var}{\mbox{\rm Var}}

\newcommand{\UU}{\mathcal U}
\newcommand{\EE}{\mbox{$\mathcal{E}$}}

\newcommand{\NN}{\mathcal N}

\newcommand\1{\leavevmode\hbox{\rm \small1\kern-0.35em\normalsize1}}

\title[Sobol indices and stochastic ordering of model parameters.]{On the consistency of Sobol indices with respect to stochastic ordering of model parameters.}

\author{A. Cousin}
\address{Université de Lyon, Université Lyon 1, Laboratoire SAF EA 2429}
\email{areski.cousin@univ-lyon1.fr}

\author{A. Janon}
\address{Université Paris Sud}
\email{alexandre.janon@math.u-psud.fr}

\author{V. Maume-Deschamps}
\address{Université de Lyon, Université Lyon 1, Institut Camille Jordan ICJ UMR 5208 CNRS}
\email{veronique.maume@univ-lyon1.fr}

\author{I. Niang}
\address{Université de Lyon, Université Lyon 1, Laboratoire SAF EA 2429}
\email{}

\begin{document}
\maketitle
\begin{abstract}
In the past decade, Sobol's variance decomposition have been used as a  tool - among others -  in risk management (\cite{Borgo,shapley}).  We show some links between global sensitivity analysis and stochastic ordering theories. This gives an argument in favor of using Sobol's indices in uncertainty quantification, as one indicator among others. 
\end{abstract}

\section*{Introduction}


Many models encountered in applied sciences involve input parameters which are often not precisely known. Sobol indices are used so as to assess the sensibility of a model output as a function of its input parameters. In other words, they quantify the impact of inputs' uncertainty on an output. Sobol indices are widely used for example in hydrology (see \cite{hydro1,hydro2}). Recently (\cite{Borgo,shapley,bouquin_sensi}), this global sensitivity analysis has been used as a risk management tool, amongst other indicators (\cite{borgo2}). In this work, our goal is to prove that Sobol indices behave coherently with respect to stochastic ordering theory. This question arises naturally: under some reasonable conditions on the output, uncertainty quantifiers should increase if the uncertainty on the input increases in some way. Thus, our problematic is to find for which kind of stochastic orders and under which sufficient conditions on the output function, the Sobol indices behave coherently. \\
Roughly speaking, given two random variables $X$ and $Y$, the $X$-Sobol indeces on $Y$ is given by
$$S_X = \frac{\var(\E(Y\ | \ X))}{\var(Y)}\/.$$
It is a statistical indicator of the relative impact of $X$ on the variability of $Y$. If we study the impact of several independent variables $X_1\/,\ldots\/, X_k$ on $Y$, then the Sobol indices may be used to provide a hierarchization of the $X_i$'s with respect to their impact on $Y$. It is an interesting alternative to regression coefficient, which may be hardly interpreted if the relationship between $Y$ and the $X_i$'s is far from linear. \\
\\
In our context, $X=(X_1\/,\ldots\/, \ X_k)$ is a random vector of $\R^k$, with the $X_i$'s being independent.  We are interested in the variance and then the Sobol indices of an output function $f$. We shall assume properties such as convexity and/or monotonicity of the function $f$. One of our main result is that Sobol indices have a behavior which is compatible with respect to the excess wealth order / or the dispersive order (it depends on the convexity properties of $f$, see Theorems \ref{theo:nointer} and \ref{theo:product}). 
The fact that Sobol index are in accordance with the excess wealth or  with the dispersive order confirms that it could be used to quantify some uncertainty, even if, depending on the purpose, moment-independent approaches should be prefered to variance decomposition (see \cite{borgo3}). Neverthelss, as we shall see in the examples,  the ordering of the Sobol indices heavily depends on the law of the parameters, so that one has to be careful on the conclusions.  \\
\ \\
To simplify notations, if $i\in \{1\/,\ldots k\}$, we shall write $X^{-i}$ for the random vector $(X_1\/,\ldots \/,X_{i-1}\/,X_{i+1}\/,\ldots \/, X_k)$, and if $\alpha\subset \{1\/,\ldots k\}$, we shall write $X_\alpha$ for the random vector $(X_i\/, \ i\in \alpha)$ and $X^{-\alpha}$ for the random vector $(X_i \/, \ i\not\in\alpha)$.\\
\ \\
The paper is organized as follows. In Section \ref{sec:variance}, we study the impact of stochastic orders on the variance. In Section \ref{sec:Sobol}, we state  our may results concerning the accordance of Sobol indices with respect to the dispersive order. Finally, in Section \ref{sec:examples} we provide some examples of illustrations.  In Section \ref{sec:concl}, we give some concluding remarks.
\section{Impact of the stochastic orders on the variance}\label{sec:variance}
Let us recall some particular notions of ordering on random variables / vectors. We refer to \cite{mul, Shaked-Shanthikumar} or \cite{Denuit2005} for a detailed review on stochastic orders, their relationships and properties. 
\subsection{Stochastic orders}
We shall be mainly interested in the stochastic order, the convex order, the dispersive order, the excess wealth order, the $*$ order and the Lorenz order.  For a random variable $X$, $F_X$  denotes its distribution function, and $F_X^{-1}$ the generalized inverse of $F_X$ (or the quantile function).  The survival function is $\overline{F}_X=1-F_X$.
\begin{defi}
Let $X$ and $Y$ be two random variables, we say that
\begin{enumerate}
\item $X$ is smaller than $Y$ for the standard stochastic order ($X\leq_{\mbox{st}} Y$) if and only if, for any bounded and non decreasing function $f$, 
$$\E(f(X)) \leq \E(f(Y))\/.$$ 
\item $X$ is smaller than $Y$ for the convex order ($X\leq_{\mbox{cx}} Y$) if and only if, for any bounded convex function $f$, 
$$\E(f(X)) \leq \E(f(Y))\/.$$ 
\item If $X$ and $Y$ have finite means, then $X$ is smaller than $Y$ for the dilatation order ($X\leq_{\mbox{dil}} Y$) if and only if 
$$(X-\E(X))\leq_{\mbox{cx}}(Y-\E(Y))\/.$$
\item $X$ is smaller than $Y$ for the dispersive order ($X\leq_{\mbox{disp}} Y$) if and only if $F_Y^{-1}-F_X^{-1}$ is non decreasing. 
\item  If $X$ and $Y$ have finite means, then $X$ is smaller than $Y$ for the excess wealth order ($X\leq_{\mbox{ew}} Y$) if and only if for all $p\in]0\/,1[$,
$$\int\limits_{[F_X^{-1}(p)\/,\infty[} \overline{F}_X (x)\/dx \leq \int\limits_{[F_Y^{-1}(p)\/,\infty[} \overline{F}_Y (x)\/dx \/.$$
\item If $X$ and $Y$ are non negative, $X$ is smaller than $Y$ for the star order ($X\leq_{\mbox{*}} Y$) if and only if 
$$\frac{F_Y^{-1}}{F_X^{-1}} \ \mbox{is non decreasing}\/.$$
\item If $X$ and $Y$ are non negative with finite mean, $X$ is smaller than $Y$ for the Lorenz ($X \leq_{\mbox{Lorenz}} Y$) if and only if 
$$\frac{X}{\E(X)} \leq_{\mbox{cx}} \frac{Y}{\E(Y)} \/.$$
\end{enumerate}
\end{defi}
\begin{remark}
The {\em st} and {\em cx} orders may be defined in the same way for random vectors.
\end{remark}
\subsection{Some relationships between variance and stochastic orders}
It is well known that the stochastic order and the convex order are not location-free and may not be compatible with the variance. The dispersive and excess wealth orders are location-free and in accordance with the variance. Below, we give some conditions implying some accordance of the stochastic order or the convex order with respect to the variance. 
\begin{prop}
Let $i\in \{1\/,\ldots\/, k\}$. If $X_i^*$ is a random variable, we shall write $X^{i*}$ for the random vector of $\R^k$: $X^{i*} = (X_1\/, \ldots \/,$  $X_{i-1}\/,$  $X_i^*\/,$  $X_{i+1}\/, \ldots\/, X_k)$. Let $X_i^*$ be a random variable, independent of $X$. The following holds:
\begin{enumerate}
\item If $X_i^*\leq_{\mbox{st}} X_i$ then $X^{i*}\leq_{\mbox{st}} X$. In particular, if $f$ is non decreasing with respect to its $i$-th component and $\E(f(X^{i*}))=\E(f(X))$ then $\var(f(X^{i*})) \leq \var(f(X))$.
\item If $X_i^*\leq_{\mbox{cx}} X_i$ then $X^{i*}\leq_{\mbox{cx}}X$. In particular, if $f$ is convex with respect to its $i$-th component and $\E(f(X^{i*}))=\E(f(X))$ then $\var(f(X^{i*})) \leq \var(f(X))$.
\end{enumerate}
\end{prop}
\begin{proof}
The result follows from the definitions and the fact that for any function $f~:~\R^k~\longrightarrow~\R$,
$$\E(f(X))= \E\left(\E(f(X_1\/,\ldots \/, X_k) | X^{-i}) \right) \/.$$
\end{proof}
In what follows, we will consider the excess wealth order, for which we can prove the ordering of Sobol's indices. Let us remark that the dispersive order implies the excess wealth order. Natural examples of random variables ordered with respect to the dispersive order will be recalled in Section \ref{sec:examples}. The following results  proved in \cite{Shaked-Shanthikumar},  show that the excess wealth order (and thus the dispersive order)  is in accordance with the variance. 
\begin{prop} \cite{Shaked-Shanthikumar}\label{prop:ew}
Let $X$ and $Y$ be two random variables with finite means.
\begin{enumerate}
\item If $X\leq_{\mbox{disp}} Y$ then $X\leq_{\mbox{ew}} Y$ then $X\leq_{\mbox{dil}} Y$ and thus if $X$ and $Y$ admit an order $2$ moment, $\var(X)\leq \var(Y)$. 
\item If $X$ and $Y$ are non negative  and $X\leq_{\mbox{*}} Y$ then $X\leq_{\mbox{Lorenz}} Y$ and then 
$$\frac{\var(X)}{ \E(X)^2} \leq \frac{\var(Y)}{\E(Y)^2}\/.$$
\item If $X$ and $Y$ are non negative then  $X\leq_{\mbox{*}} Y$ if and only if $\log X\leq_{\mbox{disp}} \log Y$.
\item If $X\leq_{\mbox{disp}} Y$ and $X\leq_{\mbox{st}} Y$ then for all non decreasing convex  or non increasing concave function $f$,  $f(X)\leq_{\mbox{disp}}f(Y)$.
\item If $X$ and $Y$ are continuous random variables with supports  bounded from below by (resp.) $\ell_*$ and $\ell$, $X\leq_{\mbox{ew}} Y$ and $-\infty<\ell_*\leq \ell$,   then for all non decreasing and convex  function $f$, for which $f(X)$ and $f(Y)$ have finite means, we have $f(X)\leq_{\mbox{ew}}f(Y)$. [Theorem 4.2 in \cite{ew_correct}] 
\end{enumerate}
\end{prop} 
\begin{remark}
Result (5) above has been incorrectly stated in \cite{Koshar} and \cite{Shaked-Shanthikumar}, where the hypothesis on the left-end points of the supports was missing. This hypothesis is indeed required, as shown by the following example. \\
Consider $X$ which follows a uniform law on  $[1\/,1.9]$ and $Y$ which follows a uniform law on $[0,1]$. Then $X\leq_{\mbox{ew}} Y$. Let $f= \exp$ which is a convex and increasing function. We have that  $\var(f(X))\sim1.32$ and $\var(f(Y))\sim0.24$ so that $f(X)$ cannot be less than $f(Y)$ for the ew order (see (1) of Proposition \ref{prop:ew}). The correct statement and proof of (5) above may be found in \cite{ew_correct}, as well as an example showing that the left-end points of the support have also to be finite. 
\end{remark}
Below, we give two simple counter-examples that show that the hypothesis on $f$ above are necessary to get the inequality on the variance. 
\begin{ex}
Let  $X$ have uniform law on $[0\/, 1]$ and $Y$ have uniform law on $[0\/, 10]$. We consider the function $f$ such that $f(t)=t$ for $t \in [0\/,1]$ and $f(t)=1$ for  $t\geq 1$. $f$ is a non decreasing function and $X\leq_{\mbox{st}} Y$. But, $\var f(X) > \var f(Y)$.
\end{ex}
\begin{ex}
Let $X$ be such that $\P(X=0)=\frac{19}{20}$ and $\P(X=1)=\frac1{20}$ and $Y$ be such that $\P(Y=0)= \frac12$ and $\P(Y=10)=\frac12$. Consider, any function $f$ such that $f(0)=0$, $f(1)=10$ and $f(10)=1$. Then, we have $\E(f(X))=\E(f(Y))$, $X\leq_{cx} Y$ but $\var(f(X))>\var(f(Y))$.
\end{ex}
\section{Impact on Sobol indices}\label{sec:Sobol}
Sobol indices can be used as a tool to quantify the impact of input parameters on the output. They are more accurate than the variance in order to identify the input variables that have the most important impact on the output. Our goal is to explore how an increase of riskness (in the sense of stochastic orders) of the input parameters may have an impact on the output. We begin by recalling definitions on Sobol indices. We refer to \cite{BX,sobol} or \cite{thesis_janon} for more details on this subject.  
\subsection{Some facts on Sobol indices}
As before, we consider one output $Y=f(X_1\/,\ldots\/, X_k)$ with $X_1\/,\ldots\/,X_k$ independent random variables. In what follows, if $\alpha\subset\{1\/,\ldots\/,k\}$, $X_\alpha$ is the random vector $X_\alpha =(X_i\/, i\in\alpha)$. We shall denote $\mu_{X_\alpha}$ the law of the random vector $X_\alpha$. For $\alpha \subset\{1\/,\ldots\/,k\}$, $|\alpha|$ denotes the length of $\alpha$, i.e. its number of elements.\\
The function $f$ can be decomposed into 
\begin{equation}\label{eq:decompSobol}f(X_1\/,\ldots\/,X_k) =  \sum_{\alpha\subset \{1\/,\ldots\/,k\}} f_\alpha(X_\alpha) \/,\end{equation}
with 
\begin{enumerate}
\item $f_\varnothing = \E(f(X))$,
\item $\displaystyle \int f_\alpha \/ d\mu_{X_i}=0$ if $i\in\alpha$,
\item $\displaystyle \int f_\alpha \cdot f_\beta \/ d\mu_X =0$ if $\alpha\neq \beta$.
\end{enumerate}
The functions $f_\alpha$ are defined inductively:
$$f_\varnothing= \E(f(X))\/,$$
for $i\in \{1\/,\ldots\/,k\}$
\begin{equation}\label{eq:fi}f_i(X_i) = \E(f(X) \ |\ X_i)-f_\varnothing = \int f \/d\mu_{X^{-i}} -f_\varnothing \/.\end{equation}
If the $f_\beta$ have been defined for $|\beta|<n$, let $\alpha\subset\{1\/,\ldots \/,k\}$ with $|\alpha|=n$ then,
$$f_\alpha(X_\alpha) = \int f d\mu_{X^{-\alpha}} - \sum_{\beta \subsetneq\alpha}  f_\beta(X_\beta)  \/.$$
With these notations, we have that:
$$\var(Y) = \var(f(X)) = \sum_{\alpha\subset\{1\/,\ldots\/,k\}}\var(f_\alpha(X_\alpha))=\sum_{\alpha \subset\{1\/,\ldots\/,k\}} \E(f_\alpha(X_\alpha)^2)  \/.$$
This decomposition of variance is often called Hoeffding decomposition (\cite{van2000asymptotic}).
The impact of $X_i$ on $Y=f(X)$ may be measured by the Sobol index:
\begin{equation}
\label{Sobol_individual}
S_i = \frac{\var(\E(f(X)\ |\ X_i))}{\var(Y)} = \frac{\E(f_i(X_i)^2)}{\var(Y)}\/.
\end{equation}
There are also interactions between the variables $X_1\/,\ldots\/,X_k$, they are identified by the $f_\alpha$, with $|\alpha|\geq 2$. The total Sobol indices take into account the impact of the interactions:
\begin{equation}
\label{Sobol_total}
S_{T_i} = \frac{\displaystyle\sum_{i\in \alpha\subset\{1\/,\ldots\/,k\}} \var(f_\alpha(X_\alpha))}{\var(Y)} = \frac{\displaystyle\sum_{i\in \alpha\subset\{1\/,\ldots\/,\/k\}} \E(f_\alpha(X_\alpha)^2)}{\var(Y)}\/.
\end{equation}
\subsection{Relationship with stochastic orders when there is no interactions}\label{sec:nointer}
In this section, we assume that there is no interactions between the $X_i$'s, that is, $f(X)$ can be expressed in the following additive form:
\begin{equation}
\label{f_no_interaction}
f(X_1, \ldots, X_k) = \sum_{j=1}^k g_j(X_j) + K 
\end{equation}
where $g_1, \ldots, g_k$ are real-valued functions and $K\in \R$. It is straitghforward to prove that, in that case, decomposition (\ref{eq:decompSobol}) reduces to
\begin{equation}\label{eq:nointer}f(X) = \sum_{i=1}^k f_i(X_i) +f_\varnothing \/,\end{equation}
so that, for any $i=1,\ldots,k$, the ``individual'' Sobol index defined by (\ref{Sobol_individual}) coincides with the total Sobol index defined by (\ref{Sobol_total}).\\

As in the previous section, $X_i^*$ denotes another variable that will be compared to $X_i$. We shall assume $X_i^*\leq_{\mbox{ew}} X_i$ and study the impact of replacing $X_i$ by $X_i^*$ on Sobol indices. We assume that $X_i^*$ is independent of $X^{-i}$ and we denote by $X^* = (X_1\/,\ldots\/,X_{i-1}\/,X_i^*\/,X_{i+1}\/,\ldots\/,X_k)$ the vector $X$ where the $i$-th component has been replaced by $X_i^*$ and by
$$S_i^* = \frac{\var(\E(f(X^*)\ |\ X_i^*)))}{\var(f(X^*))}\/$$
the $i$-th Sobol index associated with $f(X^*)$.
Because we shall use the excess wealth order, we assume that $X_i$ and $X_i^*$ have finite means,  this hypothesis may be relaxed by considering random variables ordered with respect to the dispersive order. 
\begin{theorem}\label{theo:nointer}
We assume that there is no interactions, i.e. (\ref{f_no_interaction}) is satisfied. Let $X_i^*$ be a random variable independent of $X^{-i}$ and assume that $X_i^*\leq_{\mbox{ew}} X_i$  and $-\infty<\ell_*\leq \ell$, where $\ell_*$ and $\ell$  are the left-end points of the support of $X_i^*$ and $X_i$.  If $g_i$ is a non decreasing convex function, then $S_i^*\leq S_i$ and $S_j^*\geq S_j$ for $j\neq i$.
\end{theorem}
The proof of Theorem \ref{theo:nointer} makes use of Proposition \ref{prop:ew}.
\begin{proof} 
As $g_i$ is a non decreasing convex function, Proposition \ref{prop:ew} implies $g_i(X_i^*)\leq_{\mbox{ew}} g_i(X_i)$ and $\var(g_i(X_i^*))\leq \var(g_i(X_i))$. 
Now, the Hoeffding's decomposition of $f(X)$ can be expressed as (\ref{f_no_interaction}) where 
$f_\varnothing = \E[f(X)]$ and $f_j(X_j) = g_j(X_j) - \E[g_j(X_j)]$ and the Hoeffding's decomposition of $f(X^*)$ writes
$$f(X^*) = \sum_{j\neq i} f_j(X_j) + f_i^*(X_i^*) +f_\varnothing^* \/$$
where 
$f_\varnothing^* = \E[f(X^*)]$ and $f_i^*(X_i^*) = g_i(X_i^*) - \E[g_i(X_i^*)]$.

Then, from (\ref{Sobol_individual}), the $i$-th Sobol indices of $f(X)$ and $f(X^*)$ are such that 
\begin{equation}
\label{Sobol_ind_proof}
S_i = \left[1+ \frac{\displaystyle\sum_{j\neq i} \var(g_j(X_j)) }{\displaystyle\var(g_i(X_i))}\right]^{-1} \/
\end{equation}
and
\begin{eqnarray}
S_i^* &=& \frac{\var(g_i(X_i^*))}{\displaystyle \sum_{j\neq i} \var(g_j(X_j)) + \var(g_i(X_i^*))}\\
&=& \left[1+ \frac{\displaystyle\sum_{j\neq i} \var(g_j(X_j)) }{\displaystyle\var(g_i(X_i^*))}\right]^{-1} \/.
\end{eqnarray}

We have already noticed that $\var(g_i(X_i^*))\leq \var(g_i(X_i))$ and thus we conclude that $S_i^*\leq S_i$. The result for $j\neq i$ follows from the fact that $\var(g_i(X_i^*))\leq \var(g_i(X_i))$ and 
$$S_j = \frac{\var(g_j(X_j))}{\displaystyle \sum_{j\neq i} \var(g_j(X_j)) + \var(g_i(X_i))} \ \mbox{and} \ S_j^* = \frac{\var(g_j(X_j))}{\displaystyle \sum_{j\neq i} \var(g_j(X_j)) + \var(g_i(X_i^*))}\/.$$
\end{proof}
\begin{remark}\label{rem:disp}
Note that, from Proposition \ref{prop:ew}, the previous result also holds for any non decreasing convex  or non increasing concave function $g_i$ as soon as  $X_i^*\leq_{\mbox{disp}} X_i$ and $X_i^*\leq_{\mbox{st}} X_i$. In addition, it is shown in \cite{Shaked-Shanthikumar} that if $X_i^*\leq_{\mbox{disp}} X_i$ with common and finite left end points of their support (i.e., $\ell^*=\ell$) then $X_i^*\leq_{\mbox{st}} X_i$.  
\end{remark}

\subsection{Relationship with stochastic orders when there are interactions}
In the case where there are interactions, we have to consider 
 the total Sobol indices as defined by (\ref{Sobol_total}).
We will first show that the $i$-th total Sobol indices are ordered if $X_i^*\leq_{\mbox{disp}} X_i$ and $X_i^*\leq_{\mbox{st}} X_i$, provided that the function $f$ is a product of  functions of one variable whose $\log$ is non decreasing and convex. Then we consider some extensions of that case. 
\begin{theorem}\label{theo:product}
We assume that $f$ writes:
\begin{equation}
\label{Multi_form}
f(X_1\/,\ldots \/, X_k) = g_1(X_1) \times \cdots \times g_k(X_k) +K
\end{equation}
where  $K\in\R$ and $g_j$, $j=1, \ldots, k$ are real-valued functions.
Let $X_i^*$ be a random variable independent of $X^{-i}$ and assume that $X_i^*\leq_{\mbox{disp}} X_i$ and $X_i^*\leq_{\mbox{st}} X_i$. If $\log g_i$ is a  non decreasing convex or a non increasing concave function, then $S_{T_i}^*\leq S_{T_i}$ and $S_{T_j}^*\geq S_{T_j}$, for $j\neq i$.
\end{theorem}
\begin{proof}
Without loss of generality, we may assume that $K=0$. With the hypothesis of Theorem \ref{theo:product}, the decomposition (\ref{eq:decompSobol}) satisfies: for all $j=1\/,\ldots\/,k$,
$$f_j(X_j) = (g_j(X_j)-\E(g_j(X_j))\prod_{\ell\neq j}\E(g_\ell(X_\ell))\/,$$
and following, e.g. \cite{decomp}, for $\alpha\subset \{1\/,\ldots\/,k\}$, we have
\begin{equation}\label{eq:f_alpha}
 f_\alpha(X_\alpha) = \sum_{\beta \subset \alpha} (-1)^{|\alpha|-|\beta|} \E(f(X) | X_\beta) \/.
\end{equation}
The form of $f$ then gives:
\begin{eqnarray*}
f_\alpha(X_\alpha) &= &\sum_{\beta \subset \alpha} (-1)^{|\alpha|-|\beta|} \prod_{j\in\beta} g_j(X_j) \prod_{j\not\in\beta} \E(g_j(X_j)) \\
&=&\prod_{j\not\in\alpha} \E(g_j(X_j))  \prod_{j\in\alpha} \left(g_j(X_j) - \E(g_j(X_j))\right)\/.
\end{eqnarray*}
We write 
$$f_{T_i} = \sum_{ \alpha\ni i} f_\alpha \/$$
Then,
\begin{eqnarray*}
f_{T_i}& = &\sum_{\alpha\ni i} \prod_{j\not\in\alpha} \E(g_j(X_j))  \prod_{j\in\alpha} \left(g_j(X_j) - \E(g_j(X_j))\right)\\
&=&(g_i(X_i) - \E(g_i(X_i)) \sum_{\gamma \subset \{1\/,\ldots \/,k\}\setminus \{i\}} \prod_{j\not\in\gamma} \E(g_j(X_j))  \prod_{j\in\gamma} \left(g_j(X_j) - \E(g_j(X_j))\right)\/.
\end{eqnarray*}
Now, 
\begin{eqnarray*}
\lefteqn{\sum_{\gamma \subset \{1\/,\ldots \/,k\}\setminus \{i\}} \prod_{j\not\in\gamma} \E(g_j(X_j))  \prod_{j\in\gamma} \left(g_j(X_j) - \E(g_j(X_j))\right)}\\
&&=\prod_{j\in \{1\/,\ldots \/,k\}\setminus \{i\}}\left(g_j(X_j) - \E(g_j(X_j))+\E(g_j(X_j))\right)\\
&&= \prod_{j\in \{1\/,\ldots \/,k\}\setminus \{i\}}g_j(X_j) \/.
\end{eqnarray*}
So that, finally,
\begin{equation}\label{eq:f_total}
f_{T_i} (X)= \left(g_i(X_i)-\E(g_i(X_i)\right) \prod_{j\neq i} g_j(X_j) \/.
\end{equation}
We denote by $f_\alpha^*$ the functions involved in the Sobol decomposition of $f(X^*)$ and by $f_{T_i}^*$ the sum of the $f_\alpha^*$'s over the $\alpha$ for which $i\in\alpha$. Then,
$$ f_{T_i} ^*(X^*)= \left(g_i(X_i^*)-\E(g_i(X_i^*)\right) \prod_{j\neq i} g_j(X_j) \/,$$
and 
$$S_{T_i} = \frac{\var(f_{T_i}(X))}{\var(f(X))} \ \mbox{and} \ S_{T_i}^* = \frac{\var(f_{T_i}^*(X^*))}{\var(f(X^*))}\/.$$
As in the proof of Theorem \ref{theo:nointer}, this may be rewritten as
\begin{eqnarray}\label{eq:sti}
S_{T_i} &= &\left[1+ \frac{\displaystyle \sum_{\alpha\not\ni i}\var( f_\alpha(X_\alpha)}{\displaystyle \var(f_{T_i}(X))}\right]^{-1} \ \mbox{and} \\
 S_{T_i}^* &= &\left[1+ \frac{\displaystyle \sum_{\alpha\not\ni i}\var( f_\alpha^*(X_\alpha)}{\displaystyle \var(f_{T_i}^*(X^*))}\right]^{-1}\/.\nonumber
\end{eqnarray}
We have 
$$\var(f_{T_i}(X)) = \var(g_i(X_i)) \prod_{j\neq i} \E(g_j(X_j)^2)  $$
and
$$\var(f_{T_i}^*(X^*)) = \var(g_i(X_i^*)) \prod_{j\neq i} \E(g_j(X_j)^2)\/.$$
Also, if $i\not\in\alpha$,
$$\var f_\alpha(X_\alpha) = \E(g_i(X_i))^2 \var\left(\prod_{\stackrel{j\neq i}{j\not\in \alpha}} \E(g_j(X_j)) \prod_{j\in\alpha} (g_j(X_j)-\E(g_j(X_j)))\right)$$
and
$$\var f_\alpha^*(X_\alpha) = \E(g_i(X_i^*))^2 \var\left(\prod_{\stackrel{j\neq i}{j\not\in \alpha}} \E(g_j(X_j)) \prod_{j\in\alpha} (g_j(X_j)-\E(g_j(X_j)))\right).$$

So, the result follows from (\ref{eq:sti}) if 
$$\frac{\var (g_i(X_i^*))}{\E(g_i(X_i^*))^2} \leq \frac{\var (g_i(X_i))}{\E(g_i(X_i))^2}\/.$$
We have (see Proposition \ref{prop:ew})
\begin{eqnarray*}
&&\log g_i(X_i^*)\leq_{\mbox{disp}} \log g_i(X_i)  \Longleftrightarrow  g_i(X_i^*)\leq_*g_i(X_i) \\
& \text{ so that } & g_i(X_i^*)\leq_{\mbox{Lorenz}}g_i(X_i)  \ \Longrightarrow \ \frac{\var (g_i(X_i^*))}{\E(g_i(X_i^*))^2} \leq \frac{\var (g_i(X_i))}{\E(g_i(X_i))^2}\/
\end{eqnarray*}
and thus $S_{T_i}^*\leq S_{T_i}$. When $j\neq i$, we have:
$$\var(f_{T_j}(X)) = \E(g_i(X_i)^2) \var(g_j(X_j)) \prod_{p\not\in\{i\/,j\}} \E(g_p(X_p)^2) $$
and 
\begin{eqnarray*}
\var(f(X)) &= &\sum_{\alpha\subset\{1\/,\ldots\/,k\}\/, \alpha\not=\varnothing}\var(f_\alpha(X_\alpha)) \\
&=& \E(g_i^2(X_i))\sum_{\alpha\subset\{1\/,\ldots\/,k\}\setminus\{i\}\/, \alpha\not=\varnothing}\prod_{p\not\in\alpha} \E(g_p(X_p))^2\prod_{p\in\alpha} \var(g_\alpha(X_\alpha)) \\
&&+ \var(g_i(X_i)) \prod_{p\neq i} \E(g_p(X_p))^2\/.
\end{eqnarray*}
So that $S_{T_j}^*\geq S_{T_j}$ if 
$$\frac{\E(g_i(X_i^*)^2)}{\E(g_i(X_i^*))^2} \leq \frac{\E(g_i(X_i)^2)}{\E(g_i(X_i))^2} $$
which holds as above because $g_i(X_i^*)\leq_{\mbox{Lorenz}}g_i(X_i)$.
\end{proof}
The conditions on stochastic ordering between $X_i$ and $X_i^*$, and  on the log-convexity are necessary, as can be seen with the two counter-examples below.
\begin{ex}
To see the necessarity of the stochastic ordering, one can consider:
\[ f(X_1, X_2, X_3) = \exp(\exp(X_1)) \exp(X_2) \exp(X_3), \]
with $X_1$, $X_2$ and $X_3$ uniform on $[0,1]$. If  $X_1^*$ is uniform on $[1,1.9]$, then $X_1^*\leq_{\mbox{disp}} X_1$ but $X_1\leq_{\mbox{st}} X_1^*$ and it can be easily checked that $S_{T_1}^* > S_{T_1}$ ($S_{T_1}^* \approx 0.90$ and $S_{T_1} \approx 0.65$).  
\end{ex}

\begin{ex}
The log-convexity of $g_i$ is also necessary. Indeed, take:
\[ f(X_1, X_2, X_3) = g_1(X_1) X_2 X_3 \]
where 
\[ g_1(x) = \left\{ \begin{array}{ll} 0 & \text{ if } x<0.45, \\ x/10  & \text{ if }   0.45 \leq x \leq 0.5, \\ x & \text{ else. }\end{array} \right. \]
and $X_2,X_3$ uniform on $[2,3]$, $X_1$ uniform on $[0,1]$, $X_1^*$ uniform on $[0,0.5]$ so that $X_1^*\leq_{\mbox{disp}} X_1$, $X_1^*\leq_{\mbox{st}} X_1$ but $g_1$ is not log-convex. In that case, we have $S_{T_1}^* > S_{T_1}$ ($S_{T_1}^* \approx 0.99$ and $S_{T_1}\approx 0.97$).
\end{ex}

Now, we turn to the case where $f$ writes as a sum of product of convex and non decreasing functions of one variable, that is, there are: a finite set $A$ and convex and non decreasing functions $g_i^a$, $i\in\{1\/,\ldots\/,k\}$, $a\in A$, such that 
\begin{equation}\label{eq:phi_sum_prod}
f(X) = \sum_{a\in A} g_1^a(X_1) \times \cdots \times g_k^a(X_k) \/.
\end{equation}
\begin{prop}\label{prop:sum}
Assume that $f$ satisfies (\ref{eq:phi_sum_prod}). Then, for any $i\in\{1\/,\ldots\/,k\}$,
\begin{eqnarray} 
 f_i(X_i) &=&\sum_{a\in A} \left[(g_i^a(X_i)-\E(g_i^a(X_i))\prod_{j\neq i}\E(g_j^a(X_j))\right]\\
 f_{T_i}(X) &=& \sum_{a\in A} \left[(g_i^a(X_i)-\E(g_i^a(X_i)))\prod_{j\neq i} g_j^a(X_j)\right]\\
 \var(f_{T_i})&=&\sum_{a\/,b\in A} \cov(g_i^a(X_i)\/,g_i^b(X_i)) \prod_{j\neq i} \E(g_j^a(X_j) g_j^b(X_j)) \/.
 \end{eqnarray}
\end{prop}
\begin{proof}
The proof uses in a straightforward way the computations done in the proof of Theorem \ref{theo:product}. 
\end{proof}
We deduce the two following extensions  of Theorem \ref{theo:product}.
\begin{prop}\label{prop:disj}
Let $\{I_a\}_{a\in A}$ be a partition of $\{1\/,\ldots\/,k\}$ and assume that 
$$f(X) = \sum_{a\in A} \prod_{j\in I_a} g_j(X_j)$$
where the $g_j$'s are real-valued functions. Let $X_i^*$ be a random variable independent of $X$ and assume that $X_i^*\leq_{\mbox{disp}} X_i$ and $X_i^*\leq_{\mbox{st}} X_i$. If $\log g_i$ is a non decreasing convex function or a non increasing concave function, then $S_{T_i}^*\leq S_{T_i}$ and $S_{T_j}^*\geq S_{T_j}$, for $j\neq i$.
\end{prop}
\begin{proof}
The proof follows that of Theorem \ref{theo:product} in a straighforward way because the $I_a$ are disjoints. 
\end{proof}
\begin{prop}\label{prop:sum2}
Let $f(X) = \varphi_1(X_i)\displaystyle \prod_{j\neq i} g_j(X_j) + \varphi_2(X_i)$ with $\log \varphi_1$ and $\log\varphi_2$ non decreasing and convex. If 
\begin{itemize}
\item $X_i^*$ is  independent of $X$,  $X_i^*\leq_{\mbox{disp}} X_i$ and $X_i^*\leq_{\mbox{st}} X_i$. 
\item $\displaystyle\frac{\var(\varphi_2(X_i^*))}{\E(\varphi_1(X_i^*))^2}\leq \frac{\var(\varphi_2(X_i))}{\E(\varphi_1(X_i))^2}$ and $\displaystyle \frac{\cov(\varphi_1(X_i^*)\/,\varphi_2(X_i^*))}{\E(\varphi_1(X_i^*))^2}\leq \frac{\cov(\varphi_1(X_i)\/,\varphi_2(X_i))}{\E(\varphi_1(X_i))^2}$.
\end{itemize}
Then $S_{T_i}^*\leq S_{T_i}$.
\end{prop}

\begin{proof}
 Proposition \ref{prop:sum} gives
 \begin{eqnarray*}
 \var (f_{T_i}(X))& =&\var(\varphi_1(X_i))\prod_{j\neq i} \E(g_j(X_j))^2+ \var(\varphi_2(X_i))\\
 &&+ \cov(\varphi_1(X_i)\/,\varphi_2(X_i))\prod_{j\neq i} \E(g_j(X_j))
 \end{eqnarray*}
 and if $i\not\in \alpha$, 
 $$f_\alpha(X_\alpha) = \E(\varphi_1(X_i)) \prod_{j\not\in\alpha} \E(g_j(X_j)) \times \prod_{j\in\alpha} (g_j(X_j)-\E(g_j(X_j)))\/,$$
 $$\var (f_\alpha(X_\alpha)) = \E(\varphi_1(X_i))^2 \prod_{j\not\in \alpha}\E(g_j(X_j))^2 \times \var\left(\prod_{j\in\alpha}(g_j(X_j)-\E(g_j(X_j)))\right)\/.$$
  We use once more (\ref{eq:sti}) to get that $S_{T_i}^* \leq S_{T_i}$ if and only if
 {\small\begin{eqnarray*}
 \lefteqn{\frac{\displaystyle \var(\varphi_1(X_i^*))\prod_{j\neq i} \E(g_j(X_j))^2+ \var(\varphi_2(X_i^*))+ \cov(\varphi_1(X_i^*)\/,\varphi_2(X_i^*))\prod_{j\neq i} \E(g_j(X_j))}{\displaystyle \E(\varphi_1(X_i^*))^2 \sum_{\alpha \not\ni i} \prod_{j\not\in \alpha}\E(g_j(X_j))^2 \times \var\left(\prod_{j\in\alpha}(g_j(X_j)-\E(g_j(X_j)))\right)}}\\
&\leq & \frac{\displaystyle \var(\varphi_1(X_i))\prod_{j\neq i} \E(g_j(X_j))^2+ \var(\varphi_2(X_i))+ \cov(\varphi_1(X_i)\/,\varphi_2(X_i))\prod_{j\neq i} \E(g_j(X_j))}{\displaystyle \E(\varphi_1(X_i^*))^2 \sum_{\alpha \not\ni i} \prod_{j\not\in \alpha}\E(g_j(X_j))^2 \times \var\left(\prod_{j\in\alpha}(g_j(X_j)-\E(g_j(X_j)))\right)}\/.
\end{eqnarray*}}
With our hypothesis, we have that
$$\frac{\var(\varphi_1(X_i^*))}{\E(\varphi_1(X_i^*))^2}\leq \frac{\var(\varphi_1(X_i))}{\E(\varphi_1(X_i))^2} \/,$$
$$\displaystyle\frac{\var(\varphi_2(X_i^*))}{\E(\varphi_1(X_i^*))^2}\leq \frac{\var(\varphi_2(X_i))}{\E(\varphi_1(X_i))^2}$$
and
$$\displaystyle \frac{\cov(\varphi_1(X_i^*)\/,\varphi_2(X_i^*))}{\E(\varphi_1(X_i^*))^2}\leq \frac{\cov(\varphi_1(X_i)\/,\varphi_2(X_i))}{\E(\varphi_1(X_i))^2}\/.$$
This leads to the announced result.
\end{proof}

%

The second condition in Proposition \ref{prop:sum2} is very technical and unsatisfactory. Nevertheless, very simple counter-examples exist as it can be seen below. 
\begin{ex}\label{counterex}
Let $X_i\sim \UU([1.5\/,3.5])$, $X_i^*\sim\UU([0\/,1.8])$, $\varphi_1(x) = \exp(x^2)$ and $\varphi_2(x) = \exp(x)$, $g_j(x) = 1$, $j\neq i$. Then  one can show that  $X_i^*\leq_{\mbox{disp}} X_i$ and $X_i^*\leq_{\mbox{st}} X_i$. However, 
\begin{eqnarray*}
\frac{\var(\varphi_2(X_i^*))}{\E(\varphi_1(X_i^*))^2} &\sim & 0.08\\
\frac{\var(\varphi_2(X_i))}{\E(\varphi_1(X_i))^2} &\sim& 10^{-7}\\
\frac{\cov(\varphi_1(X_i^*)\/,\varphi_2(X_i^*))}{\E(\varphi_1(X_i^*))^2} &\sim & 0.3\\
\frac{\cov(\varphi_1(X_i)\/,\varphi_2(X_i))}{\E(\varphi_1(X_i))^2}&\sim & 10^{-3}\/.
\end{eqnarray*}
and, from the proof of  Proposition \ref{prop:sum2}, $S_{T_i}^*> S_{T_i}$.
\end{ex}

The following result derived in \cite{cov} is mentioned here as a related result on excess wealth orders, even if it is not sufficient to obtain a more general version of Proposition \ref{prop:sum2}. 
\begin{prop}\label{prop:cov}[Corollary 3.2 in \cite{cov}]
Let $X$ and $Y$ be two finite means random variables with supports bounded from below by $\ell_X$ and $\ell_Y$ respectively. If $X\leq_{\mbox{ew}} Y$ and $\ell_X\leq \ell_Y$ then for all non decreasing and convex functions $h_1\/, h_2$ for which $h_i(X)$ and $h_i(Y)$ $i=1\/,2$ have order two moments, 
\begin{equation}\label{eq:cov}
\cov(h_1(X)\/, h_2(X))\leq \cov(h_1(Y)\/, h_2(Y)) \/.
\end{equation}
\end{prop}

\section{Examples}\label{sec:examples}
In this section,  we illustrate the previous results on some classical financial risk models. All considered models are associated with a set of parameters. In a context where these parameters are not known with certainty (due to estimation error for instance), the global sensibility analysis is useful to assess which (uncertain) input parameters mostly contribute to the uncertainty of model output and in turns, which parameters have to be estimated with caution.
In most of our examples, we will consider truncated distribution  functions (ordered with respect to the dispersive and stochastic orders). The use of truncated distribution is motivated by the fact that the distribution of financial parameters have generally bounded support. Let us first  recall the conditions under which some particular distribution functions are ordered with respect to the dispersive order. We refer to \cite{mul} for other classes of distribution functions.

\begin{prop}
Let $X$ and $Y$ be two random variables.
\begin{enumerate}
\item If $X \sim \UU[a,b] $  and $Y \sim \UU[c,d] $, then $X$ is smaller than $Y$ for the dispersive order ($X\leq_{\mbox{disp}} Y$) if and only if
$$b-a \leq d-c.$$ 
\item If $X \sim \mathcal{E}_{}(\mu) $ and $Y\sim \mathcal{E}_{}(\lambda)$, then $X$ is smaller than $Y$ for the dispersive order ($X\leq_{\mbox{disp}} Y$) if and only if
$$\lambda \leq \mu.$$
\item If $X \sim \mathcal{N}_{}(m_1,\sigma^2)$  and $Y\sim \mathcal{N}_{}(m_2,\nu^2)$,  then $X$ is smaller than $Y$ for the dispersive order ($X\leq_{\mbox{disp}} Y$) if and only if
$$ \sigma \leq \nu.$$ 
\end{enumerate}
\end{prop}
As mentioned above, most of the numerical illustrations will be based on model parameter with truncated distribution functions. We present some properties of such  distributions.
\begin{defi}
Let $X$ be a random variable with density function $f$ and $(a,b)\in \mathbb{R}^2$.  If $F$ denotes the cumulative distribution of $X$ then the truncated distribution of $X$ on the interval $[a,b]$ is the conditional distribution of $X$ given that $a<X \leq b$. The truncated density function of $X$ is then given by  
\begin{equation}
f(x|a<X \leq b)=\frac{g(x)}{F(b)-F(a)}
\end{equation}
where $g(x)=f(x)$ for all  $x$ such that $a<x \leq b$ and $g(x)=0$ else.
\end{defi}

 In what follows, we denote by $\mathcal{N}_{T}$ and $\mathcal{E}_{T}$ the  truncated normal and the truncated exponential laws respectively.
\begin{prop}
\label{quant}
 Let $X$ and $Y$ be two random variables
  \begin{enumerate}
 \item if $X\sim \mathcal{N}_{T}(m,\sigma^2)$ where $X$ is truncated on $[a,b] $ then  the quantile function of $X$ is given by 
 $$F_X^{-1}(x)=\phi^{-1}(\phi(\alpha) + x(\phi(\beta)-\phi(\alpha)))\sigma + m$$
 where $\alpha=\frac{a-m}{\sigma}$, $\beta=\frac{b-m}{\sigma}$ and where $\phi$ is the standard normal cumulative distribution function.
 \item if  $Y\sim \mathcal{E}_{T}(\lambda)$ is truncated on  $[a,b] $ then  the quantile function of $Y$ is given by 
  $$F_Y^{-1}(x)=-\frac{1}{\lambda}\log(e^{-\lambda a} + x(e^{-\lambda b}-e^{-\lambda a})) $$
  where $\lambda$ denotes the parameter of the exponential distribution.  
 \end{enumerate}
 
\end{prop}
 In the following lemma, we give some conditions that ensure the  ordering of two truncated random variables with respect to the dispersive order.
 \begin{lemma}
 Let $X$ and $Y$ be two random variables.
 \begin{enumerate}
 \item If $X \sim \UU[a,b]$ and $Y\sim \mathcal{N}_{T}(m,\sigma^2)$ where $Y$ is truncated on $[c,d]$, then $X$ is smaller than $Y$ for the dispersive  order $X\leq_{\mbox{disp}} Y$  if and only if
 $$b-a \leq \sigma \sqrt{2 \pi}(\phi(\beta)-\phi(\alpha))$$ where $\phi$ represents the cumulative distribution of a standard gaussian law and where $\alpha$ and $\beta$ are given by $\alpha=\frac{c-m}{\sigma}$ and
$\beta=\frac{d-m}{\sigma}$.
 \item If $X \sim \mathcal{E}_{T}(\mu) $ and $Y\sim \mathcal{E}_{T}(\lambda)$ are truncated on the same interval then  $X\leq_{\mbox{disp}} Y$ if and only if $$\lambda \leq \mu.$$ 
 \end{enumerate}
 \end{lemma}
\begin{proof}
 The previous conditions can be easily derived by differentiating the difference of the quantile functions $F_Y^{-1}-F_X^{-1}$ and by using the fact that this derivative should be positive.
 \end{proof}
 
 Also, we recall (see Remark \ref{rem:disp}) that if two random variables ordered for the dispersive order have the same finite left point of their support, then they are ordered for the stochastic order.
 In the example that we will consider, $\mathcal{N}_{T}$ (resp. $\mathcal{E}_{T}$) denotes the truncated gaussian (resp. exponential) distribution on $[0,2]$.\\
  
\subsection{Value at Risk sensitivity analysis}
In risk management, the Value-at-risk (VaR) is a widely used risk measure of the risk of losses associated with portfolio of financial assets (such as stock, bond, etc). 
From a mathematical point of view, if $L$ denotes the loss associated with a portfolio of assets, then  $VaR_{\alpha}(L)$ is defined as the $\alpha$-quantile level of this loss, i.e.,
$$VaR_{\alpha}(L):=\inf\left\{x \in \mathbb{R}:F_{L}(x) \ge \alpha \right\}$$
where $F_{L}$ denotes the cumulative distribution function of $L$.\\
Let us consider  a portfolio loss of the form $L=S_{T}-K$ where $K$ is positive and where $S_T$ stands for the aggregate  value at time $T$ of a basket of financial assets. This corresponds to the loss at time $T$ of a short position on this portfolio when the latter has been sold at time $0$ for the price $K$. We assume that $S$ follows a geometric brownian motion so that its value at time $T$ can be expressed as
$$
S_T = S_0\exp\left(\mu T + \sigma W_T\right)
$$
where $W_T$ is the value at time $T$ of a standard Brownian motion, $\mu$ (resp. $\sigma$) is a positive drift (resp. volatility) parameter. 
Therefore, the $\alpha$-Value-at-Risk associated with loss $L$ is given  by 
\begin{equation}
\label{var}
VaR_{\alpha}(L)=S_{0}\exp\left( \mu T + \sigma \sqrt{T} \phi^{-1}(\alpha) \right) -K.
\end{equation}
where $S_0 \in \mathbb{R}^{*}_{+}$ and where $\phi^{-1}$ is the normal inverse cumulative distribution function of a standard gaussian random variable. Note that,  as soon as $\alpha\geq 0.5$, the VaR expression (\ref{var}) can be seen as a product of log-convex non-decreasing functions with respect to $\mu$ and $\sigma$.\\ 
Our interest is to quantify the sensitivity of the uncertain parameters $\mu$ and $\sigma$ on the Value-at-Risk ($VaR_{\alpha}(L)$) by evaluating the total Sobol indices $S_{T_{\mu}}$ and $S_{T_{\sigma}}$ as defined by (\ref{Sobol_total}). We then analyze how an increase of uncertainty in the input parameters impacts $S_{T_{\mu}}$ and $S_{T_{\sigma}}$. In the numerical illustrations, we consider a VaR associated with a risk level $\alpha = 0.9$ and for a portfolio loss with the following characteristics:
$$T=1, S_0=100, K=100.$$
Table \ref{var_res} illustrates the consistency of total Sobol indices when the distributions of input parameters are ordered with respect to the dispersive order. 

\begin{table}[h]
\centering
\begin{tabular}{|c|c|c|c|c|c|c|c|}
\hline 
 $\mu^*$ & $\mu$ & $\sigma^*$ & $\sigma$ & $S_{T_\mu}^*$ & $S_{T_\mu}$&$S_{T_\sigma}^*$&$S_{T_\sigma}$\\
\hline
$\UU[0\/,1]$ & $\UU[0\/,1]$ &$\UU[0\/,1]$ & $\UU[0\/,2]$ & $0.41$ & $0.20$ & $0.64$ & $0.87$\\
\hline
$\UU[0\/,2]$ & $\UU[0\/,2]$ &$\UU[0\/,1] $& $\mathcal{N}_{T}(0.5\/,2)$ & $0.73$ & $0.48$ & $0.36$ & $0.69$ \\
\hline
$\UU[0\/,1]$ &$\UU[0\/,1]$& $\EE_{T}(5)$ & $\EE_{T}(1)$ &$0.53$ &$0.4$ &$0.52$ & $0.66$ \\
\hline
$\UU[0\/,1]$ & $\mathcal{N}_{T}(0.5\/,2)$ &$\UU[0\/,1]$&$\UU[0\/,1]$& $0.40$ & $0.73$ &$0.65$ & $0.35$\\ 
\hline
\end{tabular}
\caption{Total Sobol indices of VaR (\ref{var}) when $\alpha=0.9$. All digits are significant with a $95\%$ probability.}
\label{var_res}
\end{table}
Each line of Table \ref{var_res} corresponds to a scenario where one of the parameter has been increased with respect to both the dispersive and the stochastic order.
As can be seen, changing the  laws of model parameters $\mu$ and $\sigma$ have a significant impact on the values of total Sobol indices $S_{T_{\mu}}$  and $S_{T_{\sigma}}$. Note that the ordering among  Sobol indices is fully consistent with the one predicted by Theorem \ref{theo:product}.

\subsection{Vasicek model}
 In risk management, present values of financial or insurance products are computed by discounting future cash-flows. In market practice, discounting is done by using the current yield curve, which gives the offered interest rate as a function of the maturity (time to expiration) for a given type of debt contract.\\  In the Vasicek model, the yield curve is given as an output of an instantaneous spot rate model with the following risk-neutral dynamics
 \begin{equation}\label{e:vasicek}
dr_t=a(b-r_t)dt + \sigma dW_{ t}
\end{equation}
where $a$,\ $b$ and $\sigma$ are positive constants and where $W$ is a standard brownian motion. Parameter $\sigma$ is the volatility of the short rate process, $b$ corresponds to the long-term mean-reversion level whereas $a$ is the speed of convergence of the short rate process $r$ towards level $b$. The price at time $t$ of a zero coupon bond  with maturity $T$ in such a model is given by (see, e.g., \cite{Brigo2006}):
\begin{equation}
\label{Survival_Probe}
P(t,T)=A(t,T)e^{-r_tB(t,T)}
\end{equation}
where 
$$A(t,T)=\exp \left( (b -\frac{\sigma^2}{2a^2} )(B(t,T) -(T-t)) -\frac{\sigma^2}{4a}B^{2}(t,T) \right)$$
and
$$B(t,T)=\frac{1-e^{-a(T-t)}}{a}.$$
The yield-curve can be obtained as a deterministic transformation of zero-coupon bond prices at different maturities.\\

In what follows, we quantify the relative importance of the input parameters  $\left \{a, b, \sigma \right \}$ affecting the uncertainty in the bond price at time $t=0$. In the following numerical experiments, the maturity $T$ and the initial spot rate $r_0$ are chosen such that $T=1$ and $r_0=10\%$. Table \ref{tab_vasi1} and \ref{tab_vasi2} reports the total Sobol indices of the parameter $a, b, \sigma$ under two different risk perturbations in the probability laws of these parameters. Table \ref{tab_vasi1} illustrates the effect of an increase of the mean-reverting level with respect to the dispersive order and the stochastic dominance order, i.e., $b^* \leq_{\mbox{disp}} b$ and $b \leq_{\mbox{st}} b^*$. We observe from Table \ref{tab_vasi1} that the relative importance of the mean-reverting level $b$ increases from 0.52 to 0.57. Although the total Sobol index of $\sigma$ decreases, the total index of $a$ increases. Note that the assumptions of Theorem \ref{theo:product} are not satisfied here : one can show that the output function (\ref{Survival_Probe}) is log non decreasing and log convex in $b$ but the multiplicative form (\ref{Multi_form}) does not hold. 
\begin{table}[h]
\centering
\begin{tabular}{|c|c|c||c|c|c|}
\hline 
parameter&law&total index& parameter&law&total index\\
\hline
$a$&$\UU[0\/,1]$&$0.41$& $a$ &$\UU[0\/,1]$&$0.48$\\
$b^*$&$\UU[0\/,1]$&$0.52$&$b$&$\UU[0\/,2]$&$0.57$\\
$\sigma$&$\UU[0\/,1]$&$0.18$& $\sigma$&$\UU[0\/,1]$&$0.06$\\
\hline
\end{tabular}
\caption{Total Sobol indices as a result of a risk perturbation of $b$. All digits are significant with a $95\%$ probability.}
\label{tab_vasi1}
\end{table}

Table \ref{tab_vasi2} displays the total Sobol indices when $\sigma^* \leq_{\mbox{disp}} \sigma$ and $\sigma^* \leq_{\mbox{st}} \sigma$. The law of $\sigma$ is taken as a truncated gaussian random variable on $[0,2]$ and has a variance of $0.3$. We observe an increase in the total Sobol index of $\sigma$ and a decrease in total index of $a$ and $b$.

\begin{table}[h]
\centering
\begin{tabular}{|c|c|c||c|c|c|}
\hline 
parameter&law&total index& parameter&law&total index\\
\hline
$a$&$\UU[0\/,1]$&$0.41$& $a$ &$\UU[0\/,1]$&$0.25$\\
$b$&$\UU[0\/,1]$&$0.52$&$b$&$\UU[0\/,1]$&$0.13$\\
$\sigma^*$&$\UU[0\/,1]$&$0.18$& $\sigma$&$\NN_{{\mbox{T}}}(0.5\/,2)$&$0.70$\\
\hline
\end{tabular}
\caption{Total Sobol indices as a result of a risk perturbation of $\sigma$. All digits are significant with a $95\%$ probability.}
\label{tab_vasi2}
\end{table}

%
\subsection{Heston model}
In finance, the Heston model is a mathematical model which assumes  that the stock price $S_t$ has a stochastic volatility $\sigma_t$ that follows a CIR process. The model is represented by the following bivariate system of stochastic differential equations (SDEs) (see, e.g., \cite{Heston_1993}) 
\begin{eqnarray}
dS_t &=&  (r-q)S_{t}dt  + \sqrt{\sigma_t}S_t dB_t \\
d\sigma_t &=& \kappa(\theta-\sigma_t)dt + \sigma \sqrt{\sigma_t}dW_t
\end{eqnarray}
where  $d \langle B,W \rangle_t=\rho dt$.\\The model parameters are
\begin{itemize}
\item[$\bullet$]  $r$: the risk-free rate, 
\item[$\bullet$]  $q$: the dividend rate,
\item[$\bullet$] $\kappa>0$: the mean reversion speed of the volatility,
\item[$\bullet$] $\theta>0$: the mean reversion level of the volatility,
\item[$\bullet$] $\sigma>0$: the volatility of the volatility,
\item[$\bullet$] $\sigma_0>0$: the initial level of volatility,
\item[$\bullet$] $\rho \in [-1,1]$: the correlation between the two Brownian motions $B$ and $W$.
\end{itemize}
The numerical computation of European option prices under this model can be done by using the fast Fourier transform approach developed in  \cite{car_mad}
which is applicable when the characteristic function of the logarithm of $S_t$ is known in a closed form. In this framework, the price at time $t$ of a European call option with strike $K$ and time to maturity $T$ is given by
\begin{equation}
\label{Heston_price}
C(t,K,T)=S_te^{-q\tau}P_{1}-Ke^{-r\tau}P_2
\end{equation}
where for $j=1,2$
\begin{eqnarray*}
P_j &=& \frac{1}{2} + \frac{1}{\pi}\int_{0}^{\infty}{\Re \left( \frac{e^{-i\phi \log K}f_{j}(\phi;x_t,\sigma_t)}{i\phi}\right)d\phi}\\
f_{j}(\phi;x_t,\sigma_t)&=& \exp \left(C_{j}(\phi;\tau) +D_{j}(\phi;\tau)\sigma_t +i\phi x_t \right)\\
C_j &=& (r-q)i\phi \tau + \frac{\kappa \theta}{\sigma^2} \bigg[(b_j -\rho \sigma i\phi +d_j)\tau -2\log \left(\frac{1-g_{j}e^{d_j\tau}}{1-g_j} \right) \bigg ]\\
D_j &=& \frac{b_j -\rho \sigma i\phi +d_j}{\sigma^2} \left(\frac{1-e^{d_j\tau}}{1-g_je^{d_j\tau}} \right)\\
g_j &=& \frac{b_j -\rho \sigma i\phi +d_j}{b_j -\rho \sigma i\phi -d_j}\\
d_j &=& \sqrt{(\rho \sigma i\phi -b_j)^2-\sigma^2(2u_j i\phi-\phi^2)}\\
u_1&=& \frac{1}{2}, u_2=-\frac{1}{2}, b_1=\kappa-\rho \sigma, b_2=\kappa, x_t=\log S_t, \tau=T-t.
\end{eqnarray*}
Given that the input parameter are not known with certainty, which one mostly affect the uncertainty of the output pricing function (\ref{Heston_price})? Table \ref{tab_hest} displays the total Sobol indices of each parameter under two assumptions on the distribution of input parameter. In the first case (3 first columns), all parameters are assumed to be uniformly distributed. In the second case (3 last columns), we only change the distribution of the interest rate parameter $r$ is such as way that $r^* \leq_{\mbox{disp}} r$ and $r^* \leq_{\mbox{st}} r$.
%
 The option characteristics are taken as follows: $T=0.5$, $S_0=100$, $K=100$.
\begin{table}[h]
\centering
\begin{tabular}{|c|c|c||c|c|c|}
\hline 
parameter&law&total index& parameter&law&total index\\
\hline
$r^*$&$\UU[0\/,1]$&$0.32$&$r$&$\UU[0\/,2]$&$0.73$\\
$q$&$\UU[0\/,1]$&$0.39$&$q$&$\UU[0\/,1]$&$0.20$\\
$\kappa$&$\UU[0\/,1]$&$0.0036$&$\kappa$&$\UU[0\/,1]$&$0.0009$\\
$\theta$&$\UU[0\/,1]$&$0.0082$&$\theta$&$\UU[0\/,1]$&$0.0020$\\
$\sigma$&$\UU[0\/,1]$&$0.0012$& $\sigma$&$\UU[0\/,1]$&$0.0004$\\
$\sigma_0$&$\UU[0\/,1]$&$0.30$& $\sigma_0$&$\UU[0\/,1]$&$0.08$\\
$\rho$&$\UU[0\/,1]$&$0.0011$& $\rho$&$\UU[0\/,1]$&$0.0004$\\
\hline
\end{tabular}
\caption{Total Sobol indices for the price of a call option in the Heston model. All digits are significant with a $95\%$ probability.}
\label{tab_hest}
\end{table}\\
As can be observed in Table  \ref{tab_hest}, the influence of the input factors $\kappa, \theta, \sigma$ and $\rho$ is negligible under the two considered assumptions. Note that most of the uncertainty in the option price is due to the dividend yield $q$ and the interest rate $r$. Similar conclusions are outlined in \cite{bouquin_sensi} but the difference here is that we analyze how the total Sobol indices are affected by a change in the law of some input parameters. Interestingly, when the law of the  interest rate parameter $r^*$ is  changed to $r$, this parameter becomes more important than the dividend yield in terms of output uncertainty.


\section{Concluding remarks} \label{sec:concl}
We have enlightened the fact that Sobol indices are compatible with the stochastic orders theory, and more precisely with the excess wealth order, or the dispersive order, provided that the output function satisfies some monotonicity and convexity properties.  The Vasicek and Heston model examples suggest that the hypothesis on the form of the output function $f$ might be relaxed (hypothesis of Theorem \ref{theo:product} are not fulfilled, especially in the Heston case). In other words, the compatibility between Sobol indices and stochastic orders should hold for more general functions than those considered in the present paper. On an other hand, as shown in the examples, the choice of the input laws is crucial: the most influential parameters (in the sense that it corresponds to the greater Sobol index) may change as some input distributions are perturbed. A way to overcome this difficulty could be to use our result on the consistency of Sobol index with stochastic orders to get universal bounds on Sobol indices for a given class of input laws.  
 \bibliographystyle{plain}
\bibliography{biblio_sensi}

\begin{thebibliography}{10}

\bibitem{borgo3}
E.~Borgonovo.
\newblock Measuring uncertainty importance: Investigation and comparison of
  alternative approaches.
\newblock {\em Risk Analysis}, 26(5):1349--1361, 2006.

\bibitem{borgo2}
E.~Borgonovo.
\newblock A new uncertainty importance measure.
\newblock {\em Reliability Engineering \& System Safety}, 92(6):771--784, 2007.

\bibitem{Borgo}
E.~Borgonovo and L.~Peccati.
\newblock {\em On the quantification and decomposition of uncertainty},
  volume~41.
\newblock Springer, 2007.

\bibitem{Brigo2006}
M.~Brigo and F.~Mercurio.
\newblock {\em Interest Rate Models - Theory and Practice}.
\newblock Springer, 2006.

\bibitem{BX}
G.~T. Buzzard and D.~Xiu.
\newblock Variance-based global sensitivity analysis via sparse-grid
  interpolation and cubature.
\newblock {\em Commun. Comput. Phys.}, 9(3):542--567, 2011.

\bibitem{car_mad}
P.~Carr and D.Madan.
\newblock Option valuation using the fast fourier transform.
\newblock {\em computational finance}, 1999.

\bibitem{Denuit2005}
M.~Denuit, J.~Dhaene, M.~Goovaerts, and R.~Kaas.
\newblock {\em Actuarial Theory for Dependent Risks: Measures, Orders and
  Models}.
\newblock Wiley, 2005.

\bibitem{cov}
E.~Fagiuoli, F.~Pellerey, and M.~Shaked.
\newblock A characterization of the dilation order and its applications.
\newblock {\em Statistical papers}, 40:393--406, 1999.

\bibitem{Heston_1993}
S.L. Heston.
\newblock A closed-form solution for options with stochastic volatility with
  applications to bond and currency options.
\newblock {\em Review of financial studies}, 6(2):327--343, 1993.

\bibitem{thesis_janon}
A.~Janon.
\newblock {\em Analyse de sensibilit{\'e} et r{\'e}duction de dimension.
  Application {\`a} l'oc{\'e}anographie.}
\newblock PhD thesis, Universit{\'e} de Grenoble, 2012.

\bibitem{Koshar}
S.~Kochar, X.~Li, and M.~Shaked.
\newblock The total time on test transform and the excess wealth stochastic
  orders of distributions.
\newblock {\em Adv. Appl. Prob.}, 34:826--845, 2002.

\bibitem{decomp}
F.~Y. Kuo, I.~H. Sloan, G.~W. Wasilkowski, and H.~Wozniakowski.
\newblock On decompositions of multivariate functions.
\newblock {\em Mathematics of computation}, 79(270):953--966, 2010.

\bibitem{mul}
A.~M{\"u}ller and D.~Stoyan.
\newblock {\em Comparison Methods for Stochastic Models and Risks.}
\newblock Wiley, 2002.

\bibitem{shapley}
Art~B. Owen.
\newblock Sobol' indices and shapley value.
\newblock 2013.

\bibitem{hydro2}
F.~Pappenberger, K.~Beven, M.~Ratto, and P.~Matgen.
\newblock Multi-method global sensitivity analysis of flood inundation models.
\newblock {\em Advances in water resources}, 31:1--14, 2008.

\bibitem{bouquin_sensi}
A.~Saltelli, M.~Ratto, T.~Andres, F.~Campolongo, J.~Cariboni, D.~Gatelli,
  M.~Saisana, and S.~Tarantola.
\newblock {\em Global Sensitivity Analysis: The Primer}.
\newblock Wiley, 2008.

\bibitem{Shaked-Shanthikumar}
M.~Shaked and J.G. Shanthikumar.
\newblock {\em Stochastic orders}.
\newblock Springer, 2007.

\bibitem{ew_correct}
M.~Shaked, M.A. Sordo, and A.~Suarez-Llorens.
\newblock A class of location-independent variability orders, with
  applications.
\newblock {\em J. Appl. Prob.}, 47:407--425, 2010.

\bibitem{sobol}
I.M. Sobol.
\newblock Global sensitivity indices for nonlinear mathematical models and
  their {Monte Carlo} estimates.
\newblock {\em Mathematics and computers in simulation}, 55:271--280, 2001.

\bibitem{van2000asymptotic}
A.~W. van~der Vaart.
\newblock {\em Asymptotic statistics}, volume~3 of {\em Cambridge Series in
  Statistical and Probabilistic Mathematics}.
\newblock Cambridge University Press, Cambridge, 1998.

\bibitem{hydro1}
H.~Varella, M.~Gu{\'e}rif, and S.~Buis.
\newblock Global sensitivity analysis measures the quality of parameter
  estimation: The case of soil parameters and a crop model.
\newblock {\em Environmental Modelling \& Software.}, 25:310--319, 2010.

\end{thebibliography}

 \end{document}